\documentclass[10pt,a4paper]{article}
\usepackage{amsmath,amssymb,amsfonts,amsthm}
\usepackage{float,graphicx,color}
\usepackage[all,cmtip]{xy}
\usepackage{tikz}
\usetikzlibrary{shapes.geometric}
\usepackage{tikz-cd}
\usetikzlibrary{patterns, arrows.meta, calc}
\usetikzlibrary{matrix}

\newcommand{\rmd}{\mathrm{d}}

\newcommand{\rmD}{\mathrm{D}}
\newcommand{\rmH}{\mathrm{H}}

\newcommand{\rmT}{\mathrm{T}}

\newcommand{\bbN}{\mathbb{N}}

\newcommand{\bbR}{\mathbb{R}}

\newcommand{\calD}{\mathcal{D}}

\newcommand{\calF}{\mathcal{F}}

\newcommand{\calH}{\mathcal{H}}

\newcommand{\calS}{\mathcal{S}}

\newcommand{\bs}{{\scriptscriptstyle{\bullet}}}

\DeclareMathOperator{\Div}{div}
\renewcommand{\div}{\Div}

\DeclareMathOperator{\grad}{grad}
\DeclareMathOperator{\curl}{curl}
\DeclareMathOperator{\defo}{def}

\DeclareMathOperator{\inc}{inc}

\DeclareMathOperator{\img}{img}
\DeclareMathOperator{\coker}{coker}

\DeclareMathOperator{\GL}{GL}
\DeclareMathOperator{\End}{End}
\DeclareMathOperator{\Hom}{Hom}

\newcommand{\frso}{\mathfrak{so}}

\newcommand{\urmH}{\underline{\rmH}}
\newcommand{\uurmH}{\underline{\underline{\rmH}}}

\newcommand{\Hsymdiv}{\uurmH_{\div}}

\newcommand{\Hinc}{\uurmH_{\inc}}

\newcommand{\RM}{\mathrm{RM}}

\DeclareMathOperator{\alter}{alt}

\newcommand{\sym}{\mathrm{sym}}

\newcommand{\CdR}{\mathcal{H}_{\mathrm{dR}}}

\newcommand{\beq}{\begin{equation}}
\newcommand{\eeq}{\end{equation}}

\newcommand{\mapping}[4]{
\left\{
\begin{array}{rcl}
\displaystyle #1  &\to& #2\\
\displaystyle #3  &\mapsto & #4
\end{array} \right.
}

\definecolor{myyellow}{rgb}{1, 0.85, 0.15}
\definecolor{myorange}{rgb}{0.85, 0.5, 0.15}

\definecolor{mymagenta}{rgb}{0.85, 0.15, 1}
\definecolor{mycyan}{rgb}{0.15, 1, 0.85}

\definecolor{myred}{rgb}{0.7, 0.15, 0.15}
\definecolor{mypurple}{rgb}{0.425, 0.15, 0.425}
\definecolor{myblue}{rgb}{0.15, 0.15, 0.7}
\definecolor{mybluegreen}{rgb}{0.15, 0.425, 0.425}
\definecolor{mygreen}{rgb}{0.15, 0.7, 0.15}

\definecolor{mylightestbrown}{rgb}{0.825, 0.45, 0.225}
\definecolor{mybrown}{rgb}{0.55, 0.3, 0.15}
\definecolor{mydarkbrown}{rgb}{0.275, 0.15, 0.075}

\definecolor{mylightestgray}{rgb}{0.667,0.667,0.667}
\definecolor{mylightgray}{rgb}{0.5,0.5,0.5}
\definecolor{mygray}{rgb}{0.333, 0.333, 0.333}
\definecolor{mydarkgray}{rgb}{0.167,0.167,0.167}

\newcounter{npoint}[section]

\newtheorem{theorem}{Theorem}[section]

\newtheorem{corollary}[theorem]{Corollary}
\newtheorem{proposition}[theorem]{Proposition}

\theoremstyle{remark}
\newtheorem{example}{Example}[section]

\theoremstyle{definition}
\newtheorem{remark}{Remark}[section]

\newcounter{quest}[exercise]

\newcounter{subquest}[quest]

\newcommand{\solution}[1]{
}

\newcounter{pushlevel}[theorem]

\newcommand{\push}{
\stepcounter{pushlevel}
 \vspace{-1ex} \noindent \hspace{4.5ex} \begin{minipage}[t]{\textwidth*\real{0.99} - 4.5ex}
\mbox{}\hspace{-1ex}\rule[-1.5ex]{.1pt}{1.5ex}\rule{5.3ex}{.1pt}
\vspace{-.65ex}

}

\newcommand{\unpush}{
\end{minipage}
\noindent \mbox{}\hspace{4.2ex}\rule{.1pt}{1.5ex}\rule{5.3ex}{.1pt} \hspace{\stretch{1}}
\noindent
\addtocounter{pushlevel}{-1}

}

\newcounter{theopoint}[theorem]
\newcommand{\tpoint}{
\medskip
\stepcounter{theopoint}
\noindent (\roman{theopoint}) 
}

\title{
A note on short and long exact sequences in the\\ BGG construction
of complexes from complexes
}

\author{Snorre H. Christiansen\thanks{Department of Mathematics, University of Oslo, PO Box 1053 Blindern, NO 0316 Oslo, Norway. email: {\tt snorrec@math.uio.no.}}}

\date{}

\begin{document}

\maketitle

\begin{abstract} We first show how the cohomology of some Bernstein-Gelfand-Gelfand (BGG) sequences that are important for the numerical analysis of partial differential equations, can be obtained through the construction of a long exact sequence connecting cohomology groups. Then we explain the extension of this result to the non-injective/surjective case through the systematic use of short exact sequences of complexes and their associated long exact sequences of cohomology groups. Finally an interpretation in terms of spectral sequences is given.
\end{abstract}

\bigskip



\section{Introduction}

New finite elements for elasticity have been designed using a link between the so-called elasticity complex and the de Rham complex, whose discretization is much better understood \cite{ArnFalWin06IMA1,ArnFalWin06IMA2,ArnFalWin06}. This link can be interpreted as an example of the BGG construction \cite{BerGelGel71}. For a more recent account of it, see \cite{CapSloSou01}, for an introduction, see \cite{Eas99}, and for course notes see \cite{Cap26}.

Many other differential complexes can be given such a presentation and studied in functional frameworks adapted to the needs of numerical analysis \cite{ArnHu21}.  Like \cite{CheHua26} we have borrowed the title from this reference. Numerous finite element spaces have been designed using this paradigm, and other types of numerical methods are starting to integrate the technique. It has been extended to more general situations within the BGG machinery \cite{CapHu24}. 

Concerning the above mentioned functional frameworks, for the differential operators of the elasticity complex, different Sobolev regularities on the spaces will be natural for applications in elasticity as above and in relativity as in \cite{Chr11NM,ChrHuLin26}. Even for a given application such as the Hellinger-Reissner formulation of elasticity, different regularities may suggest different finite elements, as in \cite{PecSch11}.

The BGG framework can be used to derive null-homotopies for differential complexes based on null-homotopies for the de Rham sequence, with appropriate regularities \cite{ChrHuSan20,CapHu23}.
It can also be used to elucidate some physical models. The BGG reduction provides the passage from, in particular, the Cosserat models to the Cauchy models of elasticity \cite{CapHu24}.

In this largely expository note, compared with previous presentations of the BGG construction geared towards the needs of numerical analysis, we mainly add emphasis on short exact sequences of complexes as they give rise to long exact sequences of cohomology groups. This key result can be found in any textbook on homological algebra, e.g. \cite[\S XX.2, Theorem 2.1]{Lan02} or \cite[Theorem 1.3.1]{Wei94}. In spite of the latter's warning (expressed page 11) that you should ''keep the proof to yourself!'', a full proof was provided in the brief review of homological algebra included in \cite[\S 3.1.1]{Chr09AWM}. In addition to illustrations of short and long exact sequences, we also provide an application of spectral sequences to the BGG situation.

We acknowledge that from the point of view of homological algebra specialists at the level of \cite{Wei94} (not to mention \cite{McC01}), and BGG practitioners, such as those cited above, there can be nothing fundamentally new in the concepts and type of arguments described here. Furthermore it should be noted that much of the literature on BGG is \emph{not} concerned with complexes and their cohomology, but rather with more general situations where naturality of the differential operators is key.

The goal here is to show how one can get to some central results of the recent literature in numerical analysis, in a quick and transparent way, using some standard techniques from homological algebra and differential geometry that are not necessarily widely known outside of those fields. In total we give three different proofs characterizing the cohomology of, in particular, continuous and discrete elasticity sequences. In turn, this could help identify concepts relevant to the further development of discretization techniques.

The paper is organized as follows. In \S \ref{sec:longexact} we construct, in the BGG context, a long exact sequence of cohomology groups, mimicking that obtained from a short exact sequence of complexes. In \S \ref{sec:interlude} we provide an inderlude on homological algebra and differential geometry: we give some background on mapping cones in \S \ref{sec:mapcon} and covariant exterior derivatives in \S \ref{sec:covextder}. In \S \ref{sec:bgg} we interpret the BGG reduction (without hypotheses of injectivity and surjectivity on the connecting map $S$) through a short exact sequence of complexes. In \S \ref{sec:spectral} we illustrate how a spectral sequence can be used to obtain essentially the same result on cohomology groups. In Appendix \S \ref{sec:alternative} we give an alternative proof of the main result of \S \ref{sec:longexact}.

\section{A long exact sequence\label{sec:longexact}}


Here we recall the setup of \cite{ArnHu21} and then provide, in Proposition \ref{prop:longexact} and its corollary, an alternative point of view on a key result from that paper. The basic BGG construction is as follows. We suppose we have two complexes (of real vectorspaces) linked by a commuting diagram:
\begin{equation}
\begin{tikzcd}
B^0 \ar[r,"\rmd"]                 & B^1 \ar[r,"\rmd"]                 & \ldots \ar[r,"\rmd"]  & B^i \ar[r,"\rmd"] & B^{i+1} \ar[r, "\rmd"] & \ldots \\
A^0 \ar[r,"\rmd"]  \ar[ur,"S^0"] & A^1 \ar[r,"\rmd"]  \ar[ur,"S^1"] &  \ldots \ar[r,"\rmd"]  \ar[ur, "S^{i-1}"] & A^i \ar[r,"\rmd"]  \ar[ur, "S^i"] & A^{i+1} \ar[r,"\rmd"] \ar[ur, "S^{i+1}"]  & \ldots 
  \end{tikzcd}
\end{equation}
Suppose furthermore that we have an index $j$ such that:
\begin{itemize} 
\item for $i <j$, $S^i$ is injective,
\item $S^j$ is bijective,
\item for $i >j$, $S^i$ is surjective. 
\end{itemize}
Then we get:
\begin{proposition} We have that the operator:
\begin{equation} \label{eq:dSd}
\rmd (S^j)^{-1} \rmd : B^j \to A^{j+1},
\end{equation}
induces a map:
\begin{equation}
D: \coker S^{j-1} \to \ker S^{j+1}.
\end{equation}
Furthermore we get a well defined sequence, with differentials induced by the two $\rmd$ and in addition the operator $D$:
\begin{equation}\label{eq:cokerker}
\begin{tikzcd}
\ldots \ar[r,"\rmd"] & \coker S^{j-2} \ar[r,"\rmd"] & \coker S^{j-1} \ar[r,"D"] & \ker S^{j+1} \ar[r,"\rmd"] &  \ker S^{j+2} \ar[r,"\rmd"] & \ldots
\end{tikzcd}
\end{equation}
\end{proposition}
\begin{proof}
\tpoint The map defined in (\ref{eq:dSd}) has range in $\ker S^{j+1}$ and kernel containing $\img S^{j-1}$. This sustains the definition of $D$.

\tpoint The identities that $ D \circ \rmd = 0$ at $ \coker S^{j-1}$ and $ \rmd \circ D = 0$ at $\ker S^{j+1}$, are inherited from the corresponding identities satisfied by the inducing operators.
\end{proof}
The challenge is to characterize the cohomology groups of the sequence (\ref{eq:cokerker}).

\begin{remark}[Sign conventions]\label{rem:signs}
To be more precise, in \cite{ArnHu21} the setup is rather that the operator $S$ anticommutes:
\begin{equation}
S \rmd + \rmd S = 0.
\end{equation}
Furthermore, it is sometimes supposed that the operator $S$ is obtained as a commutator, with operators $K^i : A^i \to B^i$:
\begin{equation}
S = \rmd K - K \rmd .
\end{equation}
(Notice that such a relation implies the anticommutation above).

We remark that if we alternate the signs of the differential $\rmd$ in both $A^\bs$ and $B^\bs$ we get:
\begin{equation}
S \rmd = \rmd S,
\end{equation}
so that $S$ is a morphism of complexes. Futhermore if we also alternate the signs of $K$ we get that:
\begin{equation}
S = \rmd K + K \rmd,
\end{equation}
which expresses that $K$ is a nullhomotopy for $S$. In particular $S$ induces the $0$ operator on cohomology.

Here, we follow the latter sign convention. We will for the most part not suppose that $S$ comes from a commutator. We even allow that it does not induce the $0$ operator on cohomology.
\end{remark}

The sequence in (\ref{eq:cokerker}) will be denoted $C^\bs$, adjusting indices for convenience. The situation may then be described as follows:
\begin{equation}\label{eq:setup}
\begin{tikzcd}
0 & 0 \\
C^{j-1} \ar[u] \ar[r] & C^j \ar[u] \arrow[dddrr,"\ D"] & 0 & 0 & 0 \\
B^{j-1} \ar[u] \ar[r] & B^j \ar[u] \ar[r] & B^{j+1} \ar[u] \ar[r] & B^{j+2} \ar[u] \ar[r]  & B^{j+3} \ar[u] \\ 
A^{j-2} \ar[u,"S^{j-2}"] \ar[r] & A^{j-1} \ar[u,"S^{j-1}"] \ar[r] & A^{j} \ar[u,"S^{j}"] \ar[r] & A^{j+1} \ar[u,"S^{j+1}"] \ar[r]  & A^{j+2} \ar[u,"S^{j+2}"] \\
0 \ar[u] & 0 \ar[u] & 0 \ar[u] & C^{j+1} \ar[u] \ar[r] & C^{j+2} \ar[u] \\
& & & 0 \ar[u] & 0 \ar[u]
\end{tikzcd}
\end{equation}
The exactness of the vertical complexes express that $S^i$ is injective and/or surjective, according to indices, and that $C^i$ is a cokernel or a kernel, according to indices.

For $i \leq j$ we denote by $P^i: B^i \to C^i$ the relevant projection, and for $i >j$ we denote by $I^i : C^i \to B^i$ the relevant injection.

Recall that a short exact sequence of complexes gives a long exact sequence on cohomology groups by a construction involving the snake lemma. 
The situation here is slightly different in that we complete the morphism $S^i : A^i \to B^{i+1}$ to the right, in the beginning (for $i < j$), and to the left, in the end (for $i > j$). We may therefore obtain a long exact sequence in the beginning, and a long exact sequence in the end. Remarkably we can connect these two long exact sequences into one long exact sequence:

\begin{proposition}\label{prop:longexact}
We have a long exact sequence :
\begin{equation}
\begin{tikzcd}
\calH^{j-1} C^\bs \ar[ddr,"\delta^{j-1}"]& \calH^j   C^\bs \ar[ddr,"\delta^j"]&  &  &  \\
\calH^{j-1} B^\bs \ar[u]  & \calH^j B^\bs \ar[u]  & \calH^{j+1}  B^\bs \ar[ddr,"\delta^{j+1}"]  & \calH^{j+2} B^\bs \ar[ddr,"\delta^{j+2}"] & \calH^{j+3} B^\bs  \\ 
\calH^{j-2} A^\bs \ar[u] & \calH^{j-1} A^\bs \ar[u] & \calH^{j} A^\bs \ar[u] & \calH^{j+1} A^\bs \ar[u]  & \calH^{j+2} A^\bs \ar[u]\\
 & & & \calH^{j+1} C^\bs  \ar[u]  & \calH^{j+2}  C^\bs \ar[u] 
\end{tikzcd}
\end{equation}
In this diagram the position of the cohomology groups mimicks positions in diagram (\ref{eq:setup}). The oblique arrows $\delta^i$ are all constructed by the snake lemma. The other arrows are the ones induced by $S^i$, $P^i$ and $I^i$.
\end{proposition}

\begin{proof}

\tpoint Exactness at $\calH^j C^\bs$.

Let $w\in C^j$ be such that $ \delta^j w = 0$.

Write $w = P^j v$ with $v \in B^j$, and $(S^j)^{-1} \rmd v = \rmd u$ with $u \in A^{j-1}$.

Then $v' = v - S^{j-1} u \in B^j$ satisfies $\rmd v' = 0$ and $P^j v' = w$.

\tpoint Exactness at $\calH^j A^\bs$.

Let $u \in A^j$ be such that $\rmd u = 0$ and $S^j u = \rmd v$ with $v \in B^j$.

Then $w = P^j v \in C^j $ satisfies $D w = 0$ and $\delta^j w = u$.

\tpoint Exactness at $\calH^{j+1} B^\bs$.

Let $v \in B^{j+1}$ be such that $\rmd v = 0$ and $\delta^{j+1} v = D w$ with $w \in C^j$.

Write $w = P^j v'$ with $v' \in B^j$.

Then write $S^j u= ( v- \rmd v')$ with $u \in A^j$ and remark that $\rmd u = \delta^{j+1} S^j u = 0$.

\tpoint Exactness at $\calH^{j+1} C^\bs$.

Let $w \in C^{j+1}$ be such that $\rmd w = 0$ and $I^{j+1} w = \rmd u$ with $u \in A^j$.

Write $v = S^j u \in B^{j+1}$. Then $\rmd  v = \rmd S^j u = S^{j+1} \rmd u = S^{j+1}I^{j+1} w = 0$ and $\delta^{j+1} v = w$.

\medskip

\noindent That's all.\end{proof}
\noindent A more conceptual proof of this result is provided in Appendix \ref{sec:alternative}.

We deduce the following:
\begin{corollary}\label{cor:shortexact}
We have exact sequences:
\begin{equation}
0 \to \coker \widetilde S^{i-1} \to \calH^i C^\bs \to \ker \widetilde S^i \to 0,
\end{equation} 
where $\widetilde S^i: \calH^i(A^\bs) \to \calH^{i+1}(B^\bs)$ is the map on cohomology induced by $S^\bs$.

In particular, if $\widetilde S^\bs = 0$, we get exact sequences:
\begin{equation}
0 \to \calH^i B^\bs \to \calH^i C^\bs \to \calH^i A^\bs \to 0.
\end{equation} 
\end{corollary}
This accounts for a large part of \cite[Theorem 6 page 1751]{ArnHu21}. Compared with that presentation we obtain the cohomology of the \emph{output complex} $C^\bs$ without studying the \emph{twisted complex}.

\begin{example}[Elasticity and Regge Calculus] We exemplify this technique on the elasticity complex, defined on a domain in $\bbR^3$. The Sobolev regularities are from \cite{Chr11NM}:
\begin{equation}
\label{eq:Sobolev-3D-std}
\begin{tikzcd}
0 \ar[r] & \urmH^1(\bbR^3) \ar[r,"\defo"] &  \Hinc^0(\bbR^{3 \times 3}_\sym) \ar[r,"\inc"] & \Hsymdiv^{-1}(\bbR^{3 \times 3}_\sym) \ar[r,"\div"] & \urmH^{-1}(\bbR^3) \ar[r] & 0.
\end{tikzcd}
\end{equation}
It can be deduced from the following diagram:
\begin{equation}
\begin{tikzcd}
\rmH^1 \otimes \bbR^3 \ar[r,"\grad"]  & \urmH^0_{\curl (S^1)^{-1} \curl} (\bbR^3) \ar[r,"\curl"]                 & \urmH^{-1}_{\curl (S^1)^{-1}, \div} (\bbR^3) \ar[r,"\div"]  &  \rmH^{-1} \otimes \bbR^3,\\
\rmH^0 \otimes \frso_3 \ar[r,"\grad"]  \ar[ur,"S^0"] & \urmH^{-1}_{\curl} \otimes \frso_3 \ar[r,"\curl"]  \ar[ur,"S^1"] &  \urmH^{-1}_{\div} \otimes \frso_3 \ar[r,"\div"]  \ar[ur,"S^2"] & \rmH^{-1} \otimes \frso_3.
\end{tikzcd}
\end{equation}
This diagram chase is just a variant of the one used to obtain \cite[Equation (34)]{ArnHu21}, using different Sobolev regularities, but with the same local algebraic operations $S^i$. We also use, from there, the result that $S^0$ is injective, $S^1$ is an isomorphism, and $S^2$ is surjective. Furthermore they induce the $0$ map on cohomology. We have denoted by $\frso_3$ the Lie algebra of antisymmetric $3 \times 3$ matrices. Denoting the elasticity complex with $C^i$, the above considerations give exact sequences:
\begin{equation}
0 \to \CdR^i \otimes \bbR^3 \to \calH^i C^\bs \to \CdR^i \otimes \frso_3 \to 0,
\end{equation}
expressed in terms of the de Rham cohomology groups $\CdR^\bs$ of the domain. Taking into account the short exact sequence that expresses the rigid body motions as a semidirect product:
\begin{equation}
0 \to \bbR^3 \to \RM_3 \to \frso_3 \to 0,
\end{equation} 
we get that:
\begin{equation}
  \calH^i C^\bs \approx \CdR^i \otimes \RM_3.
\end{equation}

Passing to the discrete setting, the distributional complex based on Regge elements defined in \cite{Chr11NM} can similarly be obtained from the diagram in \cite[Figure 3]{ChrHuLin26}. Thus its cohomology is determined. One needs to check that the discrete sequences $A^\bs$ and $B^\bs$ have the right cohomology (given by de Rham) and that the discrete $S^\bs$ has all the required properties. 
\end{example}

\begin{example}[Finite Element Systems] The interplay between continuous and discrete BGG constructions is illustrated also in \cite{ChrHu23}. It concerns the 2D elasticity sequences (emphasizing either stress or strain), with higher Sobolev regularity than above. In that case the above diagrams do not just consist of complexes of Sobolev spaces together with finite element subcomplexes.

Rather, the objects in the diagrams are functors and the arrows are natural transformations between these functors. More explicitely, in the framework of Finite Element Systems described in \cite{ChrHu23}, the objects are families of vector spaces attached to the cells in a mesh and equipped with restriction operators, and the arrows are families of linear maps commuting with these restrictions.  In other words the objects are presheaves of real vectorspaces on the fixed poset category defined by the mesh and the arrows are morphisms of such. They constitute an abelian category \cite[\S 1.6]{Wei94} and are therefore suitable for homological algebra.

It can also be noted that many categories that are natural from the point of view of analysis, such as Banach spaces or Hilbert spaces, are \emph{not} abelian. This sheds light on the role played by closed range hypotheses (on both $S^\bs$ and the differentials) in \cite{ArnHu21}.
\end{example}

\section{Interlude\label{sec:interlude}}

This interlude provides some background in homological algebra and differential geometry of vector bundles.

\subsection{The mapping cone\label{sec:mapcon}}
We now drop the requirement that $S^i: A^i \to B^{i+1}$ be injective or surjective according to indices. We define spaces $Z^i = A^i \oplus B^i$ and differentials $D^i: Z^i \to Z^{i+1}$ by:
\begin{equation}
  D^i = \left[\begin{matrix}
      \rmd & S^i\\
      0 & \rmd
    \end{matrix} \right]
   : \left[\begin{matrix} 
      B^i\\
      A^i
    \end{matrix} \right]
  \to \left[\begin{matrix} 
      B^{i+1}\\
      A^{i+1}
      \end{matrix} \right].
\end{equation}
The sign convention here is that $S$ anticommutes. This construction is known in homological algebra and algebraic topology as the \emph{mapping cone} \cite[\S 1.5]{Wei94}. In the context of BGG it is called the twisted complex. It can also be considered as the total complex associated with the double complex with just two rows, a thread we will pick up below, in \S \ref{sec:spectral}.

We have a short exact sequence of complexes \cite[\S 1.5.2]{Wei94}:
\begin{equation}
  0 \to B^\bs \to Z^\bs \to A^\bs \to 0.
\end{equation}
It gives rise to a long exact sequence on cohomology \cite[Theorem 1.3.1]{Wei94}. Here, the connecting morphism, deduced from the snake lemma, turns out to be the map on cohomology induced by $S$ \cite[Lemma 1.5.3]{Wei94}.

In the case where the map induced on cohomology by $S$ is zero, the long exact sequence gives short exact sequences:
\begin{equation}
  0 \to \calH^i B^\bs \to \calH^i Z^\bs \to \calH^i A^\bs \to 0.
\end{equation}
This characterizes the cohomology of $Z^\bs$ in terms of those of $A^\bs$ and $B^\bs$.

\begin{remark}\label{rem:isoexact}
We also note, for future use, that in the mapping cone construction, if the maps induced by $S$ on cohomology are instead isomorphisms, then $Z^\bs$ is an exact sequence \cite[Corollary 1.5.4]{Wei94}. 
\end{remark}

\subsection{Covariant exterior derivative\label{sec:covextder}}
Here we introduce some concepts from differential geometry, in particular the covariant exterior derivative associated with a connection \cite[Chapter II]{KobNom63}. In the context of numerical analysis of PDEs, such considerations were featured in \cite{ChrWin06,ChrHal12JMP} (concerning the Yang-Mills equations). Here, the flat (i.e. curvature free) case is especially important.
These remarks amount to a simplified presentation of a special case of \cite{CapSloSou01}. 

We suppose that we work on a base manifold $M$, above which we have finite dimensional fixed fibers $E$ and $F$. The space of differential $i$-forms on $M$, i.e. sections of the $i$-th exterior power of the cotangent bundle, will be denoted:
\begin{equation}
  \Lambda^i = \Lambda^i(M) = \Gamma(\Lambda^i(\rmT^\star M)).
\end{equation}
We have followed notations from the numerical analysis literature. The differential geometry literature often prefers $\Omega^i(M)$ for the object on the right side. The spaces $A^i$, $B^i$ have the form:
\begin{align}
  A^i & = \Lambda^i \otimes E,\\
  B^i & = \Lambda^i \otimes F.
  \end{align}
The spaces are equipped with the standard exterior derivative (tensored with the identity of $E$ and $F$ respectively), denoted $\rmd$.
We may also describe this situation as saying that we have the trivial connection:
\begin{equation}
  \nabla:  \Lambda^0 \otimes (E \oplus F) \to \Lambda^1 \otimes (E \oplus F),
\end{equation}
The associated covariant exterior derivative is the standard exterior derivative (tensored with the identity of $E \oplus F$), which we also just denote by $\rmd$.

Now, given the connection $\nabla$, we can define a new\footnote{We hope the two uses of the letter $A$ is not too confusing.} connection by adding any connection one-form $A$:
\begin{equation}
  A \in \Lambda^1 \otimes \End(E \oplus F).
\end{equation}
We consider the special case where the connection one-form $A$ has the form:
\begin{equation}\label{eq:as}
  A =  \left[\begin{matrix}
      0 & S\\
      0 & 0
    \end{matrix} \right]
  : \left[\begin{matrix} 
      F \\
      E
    \end{matrix} \right]
  \to \left[\begin{matrix} 
      F\\
      E 
 \end{matrix} \right].
\end{equation}
Here we are given some
\begin{equation}\label{eq:primordialS}
S \in  \Lambda^1 \otimes \Hom(E,F).
\end{equation}
This map $S$ corresponds to the previously introduced map $S^0: A^0 \to B^1$, expressed as a map $\Lambda^0 \otimes E \to \Lambda^1 \otimes F$, which at any point $x\in M$ maps the fiber $E$ to $\Lambda^1(\rmT_x^\star M) \otimes F$. Actually we consider mainly the case where it is constant on a vectorspace $M$, so it can be identified with a linear map $E \to \Lambda^1(M) \otimes F$ between finite dimensional spaces.

\begin{example}
  The cases discussed in \cite[\S 4.2]{ArnHu21} correspond to letting $M$ be a vectorspace, fixing some $l \in \bbN$ and defining $E = \Lambda^{l+1}(M)$, $F = \Lambda^{l}(M)$ and using the most natural map:
  \begin{equation}
    S: \Lambda^{l+1}(M) \to \Lambda^1(M) \otimes \Lambda^{l}(M).
  \end{equation}
The elasticity complex (\ref{eq:Sobolev-3D-std}) discussed above is the case $l = 1$ and $\dim M = 3$ of this construction, for specific choices of Sobolev norms.
\end{example}

Given such a connection one-form $A$, we deduce first the new connection $\nabla + A$, then associate with the latter the new covariant exterior derivative, denoted $\rmd_A$. It is given by:
\begin{equation}
  \rmd_A : u \mapsto \rmd u + A \wedge u.
\end{equation}
When $A$ has the above form (\ref{eq:as}), the map $u \mapsto A \wedge u$ has the form:
\begin{equation} 
  \left[\begin{matrix}
      0 & S^i\\
      0 & 0
    \end{matrix} \right]
  : \left[\begin{matrix} 
      \Lambda^i \otimes F \\
       \Lambda^i \otimes E
    \end{matrix} \right]
  \to \left[\begin{matrix} 
       \Lambda^{i+1} \otimes F\\
       \Lambda^{i+1} \otimes E 
 \end{matrix} \right].
\end{equation}
The map $S^i: A^i \to B^{i+1}$ appearing here, if it commutes with the exterior derivative, can be used in the above mapping cone on $A^\bs \oplus B^\bs$, and the differential of the mapping cone is then nothing but $D^i = \rmd_A$. Recall that $\rmd_A \circ \rmd_A = 0$ iff the curvature of $A$ is zero. More generally, denoting the curvature by $\calF(A)$:
\begin{equation}
  \rmd_A \rmd_A u = \calF(A) \wedge u, \textrm{ with } \calF(A) = \rmd A + 1/2 [A, A].
\end{equation}

Recall that connection one-forms can be gauge transformed. Given a function $Q: M \to \GL(E \oplus F)$, we transform:
\begin{equation}
  A \to A' = Q A Q^{-1} - (\rmD Q) Q^{-1}.
\end{equation}
In the special case where $Q$ is of the form:
\begin{equation}
  Q(x) = \left[\begin{matrix}
      I_F & R(x)\\
      0 & I_E
    \end{matrix} \right]
  = \exp \left ( \left [\begin{matrix}
      0 & R(x)\\
      0 & 0
    \end{matrix} \right] \right ),
\end{equation}
we have:
\begin{equation}
  Q(x)^{-1} = \left[\begin{matrix}
      I_F & - R(x)\\
      0 & I_E
    \end{matrix} \right].
\end{equation}
and:
\begin{equation}
  \rm D Q(x) = \left[\begin{matrix}
      0 & \rmD R(x)\\
      0 & 0
    \end{matrix} \right].
\end{equation}
If we gauge transform $A= 0$ to $A'$ using this $Q$ we get simply:
\begin{equation}
  A'  =  \left[\begin{matrix}
      0 & \rmD R\\
      0 & 0
    \end{matrix} \right].
\end{equation}
If now $M$ is a vectorspace and the primordial $S$ introduced in (\ref{eq:primordialS}) is constant, we can define $R$ as the contraction of $S$ by the Euler vectorfield on $M$ (defined simply by $\xi: x \mapsto \xi(x) = x$) to guarantee that $\rmD R = S$. The conclusion is that our connection deduced from $S$ can be obtained from the trivial connection by a gauge transformation. In particular it is curvature free, from the general formula:
\begin{equation}
\calF(A') = Q\calF(A) Q^{-1}.
\end{equation}
Another basic property of the gauge transformation is that:
\begin{equation}
  \rmd_{A'} (Q u) = Q (\rmd_A u).
\end{equation}
Let us write the mapping cone relative to $S$ as $Z_S^\bs$,  allowing also $S=0$ in this notation. We conclude that the induced map:
\begin{equation}
  Q : \mapping{Z_0^i}{Z_S^i}{u}{Q u},
\end{equation}
is a morphism of complexes (when $Z_0^\bs $ is equipped with $\rmd_0$ and $Z_S^\bs$ is equipped with $\rmd_S$). It is an isomorphism by the group action property of gauge transformations. In particular we get isomorphic cohomology groups.

\begin{remark} In this setting we can see that $S$ induces the $0$ map on cohomology as follows. We may arrange things as:
\begin{equation}
  \xymatrix{
    0 \ar[r] & B^\bs \ar[r] &  Z_S^\bs \ar[r] &  A^\bs \ar[r] & 0\\
    0 \ar[r] & B^\bs \ar[r] \ar[u]^{=} &  Z_0^\bs \ar[r] \ar[u]^Q &  A^\bs \ar[r] \ar[u]^{=} & 0
    }
\end{equation}
The middle upward arrow is a morphism of complexes, as noted above. The two squares commute by the upper triangular structure of $Q$ with identities on the diagonal.
Writing the long exact sequences of the two rows, connecting them by functoriality (of going from short to long exact sequences), and using that the connecting morphism on top is deduced from $S$ and that on the bottom is $0$, the claim follows. Compare with the argument given in Remark \ref{rem:signs}.
\end{remark}

\section{The BGG reduction using pseudo-inverses\label{sec:bgg}}
Compared with \cite{CapHu24} we have specialised to the case of diagrams with only two rows. We now give an account of parts of that presentation, using short exact sequences as a substitute for some direct computations that chase elements through diagrams.

Let $P_C^{i+1}$ be a projection in $B^{i+1}$ with kernel $\img S^i$ and let $P_K^{i}$ be a projection in $A^i$ with range $\ker S^i$. This amounts to choosing supplementaries of the image and kernel of $S^i$. Then we define $T^i : B^{i+1} \to A^i$ be the map obtained by inverting the induced isomorphism $S^i : \ker P^i_K \to \img S^i$ and extending it by zero on $\img P_C^{i+1}$. It is a pseudoinverse of $S^i$. Furthermore:
\begin{equation}
P_C^{i+1}= I - S^i T^i \textrm{ and } P_K^i = I - T^i S^i.
\end{equation}

We let $C^{i+1}$ denote the range of $P_C^{i+1}$ in $B^{i+1}$. It is isomorphic to the quotient space $\coker S^i = B^{i+1}/\img S^i$. The differential on $C^i$ is then $P_C^i \rmd $ where $\rmd $ is the differential of $B^i$. We also put $K^i = \ker S^i$.

We define the following endomorphism of $A^i \oplus B^i$:
\begin{equation}
 L^i  =  \left[\begin{matrix}
      I & 0\\
      -T^i \rmd  & I
    \end{matrix} \right].
\end{equation}
We introduce:
\begin{align}\label{eq:defD}
\calD & = L^{-1} \rmd_S L.
\end{align}
It is obviously a differential. We compute:
\begin{align}
\calD & = \left[\begin{matrix}
      P_C \rmd & S\\
      -P_K \rmd T \rmd  & P_K \rmd
    \end{matrix} \right].
\end{align}

We see that $K^i \oplus C^i$ is a subcomplex of $A^i \oplus B^i$ equipped with the differential $\calD$. We consider that $K^i \oplus C^i$ is equipped with the restriction of $\calD$. We let $I$ denote the inclusion map into $A^i \oplus B^i$. We notice that the restriction of $\calD$ to the subcomplex has the simpler form:
\begin{equation}\label{eq:DBGG}
\left[\begin{matrix}
      P_C \rmd & 0\\
      - P_K \rmd T \rmd  & \rmd
    \end{matrix} \right].
\end{equation}
Thus it can also be interpreted as a mapping cone. We call it the \emph{BGG reduced complex}.

From (\ref{eq:defD}) we immediately get:
\begin{equation}
\rmd_S L I = L I \calD.
\end{equation}
This can be interpreted as saying that the maps $L I$ define a morphism of complexes from $K^i \oplus C^i$ to $A^i \oplus B^i$, the former equipped with $\calD$ and the latter now equipped with $\rmd_S$. The map $LI$ is known as the splitting operator in BGG, see \cite[arxiv \S 4.3]{CapSloSou01} and \cite[\S 7.3]{Cap26}.

We also let $J^i: A^i \oplus B^i \to \img T^i \oplus \img S^{i-1}$ be the projection:
\begin{equation}
J^i : \left[\begin{matrix}
      I - P_C^i & 0\\
      0  & I - P_K^i
    \end{matrix} \right].
\end{equation}

The spaces  $\img T^i \oplus \img S^{i-1}$ are considered as a complex under the differential:
\begin{equation}\label{eq:defsi}
\calS^i = \left[\begin{matrix}
      0 & S^i\\
      0 & 0
    \end{matrix} \right].
\end{equation}
 We check by computation that:
\begin{equation}
\calS J L^{-1} = J L^{-1} \rmd_S.
\end{equation}
This gives us the sequence of complexes:
\begin{equation}\label{eq:bbgshortexact}
  \xymatrix{
    0 \ar[r] & (K^i \oplus C^i, \calD ) \ar[r]^{L I}  &  (A^i \oplus B^i, \rmd_S)  \ar[r]^{\! \! J L^{-1} \ \ \ } & (\img T^i \oplus \img S^{i-1}, \calS) \ar[r] & 0
}
\end{equation}
It is exact just because the composition $J I$ determines an exact sequence.

We summarize some of our findings as follows.
\begin{proposition}
  Equation (\ref{eq:bbgshortexact}) is an exact sequence of complexes.
\end{proposition}

We are now in a position to conclude.
\begin{proposition}\label{prop:bggiso}
The splitting map $L I$ induces isomorphisms on cohomology, from the BGG reduced complex to $A^\bs \oplus B^\bs$ equipped with $\rmd_S$.
\end{proposition}
\begin{proof}
  Since the induced maps $S^i: \img T^i \to \img S^i$, appearing in the definition (\ref{eq:defsi}) of $\calS^i$,  are isomorphisms, the third complex in (\ref{eq:bbgshortexact}) is exact.
\end{proof}
This provides an alternative proof of \cite[Theorem 2 p. 1156]{CapHu24} or \cite[Theorem p. 17]{Cap26} in the special case of two rows.

\section{The use of a spectral sequence\label{sec:spectral}}
In this section we illustrate a spectral sequence technique \cite{McC01} to obtain the cohomology of the BGG reduced complex. This computation was inspired\footnote{This section has been available as the online note \emph{Spectral sequences: friend or foe?} (Stanford University, 2008).} by \cite[\S 1.6]{Vak25}. We quote the spirit: ''What is perhaps different in this presentation is
that we will use spectral sequences to prove things that you may have already seen,
and that you can prove easily in other ways. This will allow you to get some hands-on experience in how to use them.''

Here we use:
\begin{align}
  C^{j+1} & = \coker S^j ,\\
  K^j & = \ker S^j.
\end{align}
Then we have an exact sequence of complexes:
\begin{equation}
   \xymatrix{
    0 \ar[r] & K^\bs \ar[r] &  A^\bs \ar[r]^{S^\bs} &  B^{\bs +1} \ar[r] & C^{\bs+1} \ar[r] & 0\\
   }
\end{equation}

We consider the following double complex:
\begin{equation}
\xymatrix{
0  & 0 & 0 & 0 & 0 \\
B^{j-1} \ar[u] \ar[r] & B^j \ar[u] \ar[r] & B^{j+1} \ar[u] \ar[r] & B^{j+2} \ar[u] \ar[r]  & B^{j+3} \ar[u] \\ 
A^{j-2} \ar[u]^{S^{j-2}} \ar[r] & A^{j-1} \ar[u]^{S^{j-1}} \ar[r] & A^{j} \ar[u]^{S^{j}} \ar[r] & A^{j+1} \ar[u]^{S^{j+1}} \ar[r]  & A^{j+2} \ar[u]^{S^{j+2}} \\
0 \ar[u] & 0 \ar[u] & 0 \ar[u] & 0 \ar[u] & 0 \ar[u]
}
\end{equation}
The total complex of this double complex is then the mapping cone relative to $S^i : A^i \to B^{i+1}$. It was previously denoted $Z_S^\bs$. We compute its cohomology in two ways, using spectral sequences.

First we start vertically and continue from there horizontally and then a downward knight move.  Second we start horizontally and continue from there vertically and then with an upward knight move.

\paragraph{First computation}

(i) On the first page we get:
\begin{equation}
\xymatrix{
0  & 0 & 0 & 0 & 0 \\
C^{j-1} \ar[r] & C^j  \ar[r] & C^{j+1} \ar[r] & C^{j+2} \ar[r]  & C^{j+3} \\ 
K^{j-2} \ar[r] & K^{j-1} \ar[r] & K^{j} \ar[r] & K^{j+1} \ar[r]  & K^{j+2} \\
0  & 0  & 0  & 0 & 0
}
\end{equation}

(ii) On the second page we get, for some maps $\phi^i$:
\begin{equation}
\xymatrix{
0  & 0 & 0 & 0 & 0 \\
\calH^{j-1}(C^\bs) \ar [rrd]^{\phi^{j-1}} & \calH^j(C^\bs)  \ar[rrd]^{\phi^j} & \calH^{j+1}(C^\bs) \ar[rrd]^{\phi^{j+1}} & \calH^{j+2}(C^\bs)  & \calH^{j+3}(C^\bs) \\ 
\calH^{j-2}(K^\bs)  & \calH^{j-1}(K^\bs) & \calH^{j}(K^\bs) & \calH^{j+1}(K^\bs)  & \calH^{j+2}(K^\bs) \\
0  & 0  & 0  & 0 & 0
}
\end{equation}

(iii) On the third page we get the kernels of $\phi^i$ on the second row and the cokernels on the third row. The spectral sequence has converged.

We deduce:
\begin{equation}\label{eq:spec1}
  \calH^i(Z_S^\bs) \approx  \coker \phi^{i-1} \oplus \ker \phi^i.
\end{equation}

\paragraph{Second computation} On the first page we get:
\begin{equation}
 \xymatrix{
0  & 0 & 0 & 0 & 0 \\
\calH^{j-1}(B^\bs)  \ar[u] & \calH^j(B^\bs) \ar[u] & \calH^{j+1}(B^\bs) \ar[u] & \calH^{j+2}(B^\bs) \ar[u]  & \calH^{j+3}(B^\bs) \ar[u] \\ 
\calH^{j-2}(A^\bs) \ar[u]^{\widetilde{S}^{j-2}}(A^\bs) & \calH^{j-1}(A^\bs) \ar[u]^{\widetilde{S}^{j-1}}  & \calH^{j}(A^\bs) \ar[u]^{\widetilde{S}^{j}}  & \calH^{j+1}(A^\bs) \ar[u]^{\widetilde{S}^{j+1}}  & \calH^{j+2}(A^\bs) \ar[u]^{\widetilde{S}^{j+2}} \\
0 \ar[u] & 0 \ar[u] & 0 \ar[u] & 0 \ar[u] & 0 \ar[u]
}
\end{equation} 
When the maps $\widetilde S^i$ induced on cohomology by $S^i$ are the zero maps, the second page has the same spaces, connected by the zero maps. The spectral sequence has converged.

We deduce:
\begin{equation}\label{eq:spec2}
  \calH^i(Z_S^\bs) \approx  \calH^{i}(A^\bs) \oplus \calH^i(B^\bs).
\end{equation}

\paragraph{Conclusion} Suppose now that we have anticommuting maps $\Phi^i: C^i \to K^{i+1}$ that induce, on cohomology, the maps $\phi^i: \calH^i C^\bs \to \calH^{i+1} K^\bs$ obtained above in the spectral sequence. We construct the mapping cone on the spaces $Y_\Phi^i = K^i \oplus C^i$:
\begin{equation}
  \left[\begin{matrix}
      \rmd & 0\\
      \Phi^i\ & \rmd
    \end{matrix} \right]
   : \left[\begin{matrix} 
      C^i\\
      K^i
    \end{matrix} \right]
  \to \left[\begin{matrix} 
      C^{i+1}\\
      K^{i+1}
      \end{matrix} \right]
\end{equation}
We shall prove that it has the familiar cohomology $\calH^i(A^\bs) \oplus \calH^i(B^\bs)$.

As before we have a short exact sequence:
\begin{equation}
0 \to  K^\bs \to Y_\Phi^\bs \to C^\bs \to 0
\end{equation}
The corresponding long exact sequence is:
\begin{equation}
  \xymatrix{
 \calH^{i-1}(C^\bs) \ar[r]^{\phi^{i-1}}  & \calH^{i}(K^\bs) \ar[r] &  \calH^{i}(Y_\Phi^\bs) \ar[r] & \calH^{i}(C^\bs) \ar[r]^{\phi^i} & \calH^{i+1}(K^\bs)
  }
\end{equation}
This can be shortened to:
 \begin{equation}
  \xymatrix{
 0 \ar[r] & \coker \phi^{i-1} \ar[r] &  \calH^{i}(Y_\Phi^\bs) \ar[r] & \ker \phi^i \ar[r] & 0.
  }
 \end{equation}
which gives:
 \begin{equation}
   \calH^{i}(Y_\Phi^\bs) \approx  \coker \phi^{i-1} \oplus \ker \phi^i.
 \end{equation}
By the fact that the two spectral sequences associated with the above double complex both compute the cohomology of the total complex, viz. (\ref{eq:spec1}) and (\ref{eq:spec2}), we get:
\begin{equation}\label{eq:specbgg}
    \calH^{i}(Y_\Phi^\bs) \approx  \calH^{i}(A^\bs) \oplus \calH^i(B^\bs).
 \end{equation}

When pseudoinverses of $S^i$ are introduced, as in the previous section, we may put:
\begin{equation}
 \Phi^i = - P_K^{i+1} \rmd T^i \rmd. 
\end{equation}
In view of the already noted fact that (\ref{eq:DBGG}) defines a differential, the maps $\Phi^\bs : C^\bs \to K^{\bs+1}$ anticommute. To check that this choice of $\Phi^i$ really induces $\phi^i$ on cohomology, relies on the precise definition of $\phi^i$, that is, of the differential on page 2 of the spectral sequence. The definition is explained in \cite[Exercise 5.1.2 page 121]{Wei94} (transposing from homology to cohomology). Recognizing the same zigzag, we conclude that the cohomology of the BGG reduced complex is given by (\ref{eq:specbgg}). This conclusion is a bit less precise compared with Corollary \ref{cor:shortexact} and Proposition \ref{prop:bggiso}.

\begin{remark}
Under the hypotheses of  \S \ref{sec:longexact},  each space in the mapping cone $Y_\Phi^\bs$ -- identified with the BGG reduced complex defined by (\ref{eq:DBGG}) -- will have just one component, first $C^i$ and then $K^i$. Furthermore $\Phi^i = 0$ for $i \neq j$ and $\Phi^j = -D$, the operator appearing in (\ref{eq:cokerker}).
\end{remark}

\section*{Acknowledgements}
I am grateful to Andreas Cap for interesting discussions in connection with the program at ESI (Vienna) where the course \cite{Cap26} was given, as well as comments on the first version of this manuscript. In particular this led to a better framing of the goals of BGG and the more conceptual proof of Proposition \ref{prop:longexact} provided in  Appendix \ref{sec:alternative}.

\bibliography{../Bibliography/alexandria,../Bibliography/newalexandria,../Bibliography/mybibliography}{}
\bibliographystyle{plain}

\appendix


\section{Alternative proof\label{sec:alternative}}

We now provide an alternative and more conceptual proof of the main result of \S \ref{sec:longexact}, for comparison.
\begin{proof}(of Proposition \ref{prop:longexact})
We first consider the second half of the complex (\ref{eq:setup}). It is understood that vertical complexes are extended by $0$ and are short exact. 
\begin{equation}
  \begin{tikzcd}
0\ar[r]&    B^{j+1} \ar[r]  & B^{j+2} \ar[r] & B^{j+3}  \\
0\ar[r]&    A^j  \ar[u,"S^{j}"] \ar[r] & A^{j+1}   \ar[u,"S^{j+1}"] \ar[r] & A^{j+2}  \ar[u,"S^{j+2}"]\\
    & 0 \ar[u] \ar[r] & C^{j+1}   \ar[u] \ar[r] & C^{j+2} \ar[u]  
  \end{tikzcd}
\end{equation}
We deduce the long exact sequence:
\begin{equation}\label{eq:rightlong}
  \begin{tikzcd}
    0 \ar[r] & \ker \rmd |_{A^j} \ar[r] & \ker \rmd |_{B^{j+1}} \ar[r] & \ker \rmd |_{C^{j+1}} \ar[lld] \\
      & \calH^{j+1} A^\bs \ar[r] & \calH^{j+2} B^\bs
  \end{tikzcd}
\end{equation}
This sequence can be shortened to:
\begin{equation}
  \begin{tikzcd}
    0 \ar[r] & \ker \rmd |_{A^j} \ar[r] & \ker \rmd |_{B^{j+1}} \ar[r] & C^{j+1} \cap \rmd A^j \ar[r] & 0
  \end{tikzcd}
\end{equation}
We place this sequence vertically in the first half of the complex (\ref{eq:setup}):
\begin{equation}
\begin{tikzcd}
C^{j-1} \ar[r] & C^j \arrow[r,"D"] & C^{j+1} \cap \rmd A^j \ar[r] & 0 \\
B^{j-1} \ar[u] \ar[r] & B^j \ar[u] \ar[r] & \ker \rmd |_{B^{j+1}} \ar[u] \ar[r] & 0 \\
A^{j-2} \ar[u,"S^{j-2}"] \ar[r] & A^{j-1} \ar[u,"S^{j-1}"] \ar[r] &  \ker \rmd |_{A^{j}} \ar[u,"S^{j}"]  \ar[r] & 0 
\end{tikzcd}
\end{equation}
This gives the long exact sequence:
\begin{equation}
  \begin{tikzcd}
    \calH^{j-1} A^\bs \ar[r] & \calH^j B^\bs \ar[r] & \calH^j C^\bs \ar[lld]\\
    \calH^{j} A^\bs \ar[r] & \calH^{j+1} B^\bs \ar[r] & (C^{j+1} \cap \rmd A^j)/(D C^j) \ar[r] & 0
  \end{tikzcd}
\end{equation}
Here, we want to replace $C^{j+1} \cap \rmd A^j$ by the larger space $\ker \rmd |_{C^{j+1}}$, to identify the desired cohomology group at $C^{j+1}$. We do this thanks to (\ref{eq:rightlong}) and get:
\begin{equation}
  \begin{tikzcd}
    \calH^{j-1} A^\bs \ar[r] & \calH^j B^\bs \ar[r] & \calH^j C^\bs \ar[lld]\\
    \calH^{j} A^\bs \ar[r] & \calH^{j+1} B^\bs \ar[r] & \calH^{j+1} C^\bs  \ar[lld]\\
    \calH^{j+1} A^\bs \ar[r] & \calH^{j+2} B^\bs
  \end{tikzcd}
\end{equation}
This completes the proof.
\end{proof}
  
\end{document}